\documentclass[psamsfonts]{amsart}

\usepackage{amssymb,amsfonts}
\usepackage[all,arc]{xy}
\usepackage{enumerate}
\usepackage{mathrsfs}
\usepackage[margin=1.5in]{geometry}
\usepackage[numbers]{natbib}
\usepackage{dsfont}
\usepackage{amsaddr}

\newtheorem{thm}{Theorem}[section]
\newtheorem{cor}[thm]{Corollary}
\newtheorem{prop}[thm]{Proposition}

\newtheorem{innercustomthm}{Theorem}
\newenvironment{customthm}[1]
  {\renewcommand\theinnercustomthm{#1}\innercustomthm}
  {\endinnercustomthm}

\theoremstyle{definition}

\theoremstyle{remark}
\newtheorem{rem}[thm]{Remark}

\newcommand*{\longhookrightarrow}{\ensuremath{\lhook\joinrel\relbar\joinrel\rightarrow}}
\newcommand{\Sym}{\operatorname{Sym}}
\newcommand{\Ext}{\operatorname{\bigwedge}}
\newcommand{\Tplus}{\operatorname{Tot^\oplus}}
\newcommand{\Tprod}{\operatorname{Tot^\Pi}}

\makeatletter
\let\c@equation\c@thm
\makeatother
\numberwithin{equation}{section}

\title[Hochschild Cohomology of the Exterior Algebra]{A Geometric Approach to Hochschild Cohomology of the Exterior Algebra}

\author{Michael Wong}

\address{Department of Mathematics, The University of Texas at Austin,
2515 Speedway, Austin, TX 78712}

\email{mwong@math.utexas.edu}



\begin{document}

\maketitle

\begin{abstract}
We give a new computation of Hochschild (co)homology of the exterior algebra, together with algebraic structures, by direct comparison with the symmetric algebra. The Hochschild cohomology is determined to be essentially the algebra of even-weight polyvector fields. From Kontsevich's formality theorem, the differential graded Lie algebra of Hochschild cochains is proved to be formal when the vector space generating the exterior algebra is even dimensional. We conjecture that formality fails in the odd dimensional case, proving this when the dimension is one. In all dimensions, formal deformations of the exterior algebra are classified by formal Poisson structures. 
\end{abstract}

\tableofcontents

\newpage

\section{Introduction} 

Hochschild (co)homology is a basic invariant of associative algebras that encodes, among other things, their deformation theory. Despite the ubiquity of the exterior algebra in mathematics, many of its Hochschild structures are still not fully understood or have resisted simple interpretation. As for what is known, computations of the Hochschild (co)homology groups of the exterior algebra and the cohomological ring structure have been given in various contexts : as a special case of the results for quantum complete intersections (\cite{BGMS}, \cite{BE}, \cite{Oppermann}), from direct combinatorial techniques \cite{HanXu}, and from algebraic Morse theory \cite{LamVav}. Recently, the Gerstenhaber bracket has been described for quantum complete intersections and their group extensions (\cite{GNW}, \cite{Grimley}), building off the work of \cite{BO} on twisted tensor products and of \cite{NW} on defining the bracket on complexes other than the bar complex. Furthermore, the general theory of \cite{LZZ} and \cite{Volkov} on Frobenius algebras with semisimple Nakayama automorphism ensures that Hochschild cohomology of the exterior algebra is a Batalin-Vilkovisky (BV) algebra. However, the algebraic results mentioned here are somewhat technical and do not address the question of deformations. Our purpose is to find an elementary geometric characterization of Hochschild (co)homology of the exterior algebra, together with its rich algebraic structures, and to classify formal deformations.

When equipped with its natural grading, the exterior algebra can be viewed as the algebra of functions on an affine supermanifold. The Hochschild (co)homology of this graded version, which is properly a symmetric algebra on odd generators, is well-understood from the odd variants of established geometric results. The Hochschild-Kostant-Rosenberg isomorphism \cite{HKR} identifies Hochschild homology with the space $\Omega$ of algebraic differential forms on odd affine space and Hochschild cohomology with the Gerstenhaber algebra $T_{\text{poly}}$ of polyvector fields. Under this equivalence, the divergence operator with respect to a volume form gives a BV structure. Moreover, Kontsevich's formality theorem \cite{Kont}, which provides a canonical deformation quantization of Poisson manifolds, has been adapted to the supermanifold setting \cite{CF}.  

We obtain a geometric picture for the Hochschild structures of the ungraded exterior algebra by direct comparison with its graded analog in the symmetric algebra. The means for our comparison is Koszul duality. Keller \cite{keller} showed that, if $A$ is a Koszul $\mathds{k}$-algebra and $A^!$ its Koszul dual, there exists an isomorphism of Hochschild cochains $C^*(A, A) \rightarrow C^*(A^!, A^!)$ in the homotopy category of $B_\infty$-algebras, inducing an isomorphism on cohomology $HH^*(A, A) \rightarrow HH^*(A^!, A^!)$ as Gerstenhaber algebras. Herscovich \cite{Hers} extended this result to the module structure of Hochschild homology, proving that the Tamarkin-Tsygan calculus of $A$ and $A^!$ are dual. Moreover, when $A$ is Calabi-Yau, Chen et al. \cite{CYS} enhanced Keller's Gerstenhaber isomorphism to one of BV algebras where $HH^*(A,A)$ and $HH^*(A^!, A^!)$ are given the BV differentials in \cite{Ginz} and \cite{Tradler}, respectively. The assumption of these theorems, however, is that Koszul duality exchanges ordinary and differential graded algebras: $A$ is assumed to be weight graded (Adams graded in the terminology of \cite{keller}) and concentrated in differential degree 0, while $A^!$ is concentrated along the bigrading diagonal (and has trivial differential). So the symmetric algebra $\Sym V$ generated by a finite dimensional vector space is Koszul dual to the symmetric algebra $\Sym(V^\vee (-1))$ on the shifted linear dual of $V$. In contrast, classical sources such as \cite{BGS} often identify the Koszul dual with the more elementary quadratic dual, which is weight graded and concentrated in differential degree 0. With this convention, $\Sym V$ is dual to the exterior algebra $\Ext V^\vee$, as in the BGG correspondence \cite{BGG}. 

Evidently, the quadratic dual $A^!_{\text{classical}}$ is obtained from the differential graded Koszul dual $A^!$ by subtracting the weight grading from the differential grading,
\[
\xymatrixcolsep{6pc} \xymatrix{
A \ar@{<->}[r]^-{\text{Koszul}}_-{\text{duality}} &  A^!  \ar@{<->}[r]^-{\text{diff. degree} \, -}_-{\text{weight}} & A^!_{\text{classical}}. \\ } 
\]
The exterior algebra $\Ext V$ is recovered in this way from $\Sym (V(-1))$. It is natural to ask how such a grading shift affects Hochschild (co)homology in general, and we give a partial answer in \S \ref{ho section}, \S \ref{coh section}, and \S \ref{BV structure}:

\begin{customthm}{A} 
With cohomological degree taken modulo $2$, there is an isomorphism of the odd-weight subcomplexes of Hochschild chains 
\[ C_*^{\text{\emph{odd}}}(A^!, A^!) \cong \Pi \, C_*^{\text{\emph{odd}}}(A^!_{\text{\emph{classical}}}, A^!_{\text{\emph{classical}}}) \] 
where $\Pi$ is the parity inversion operator. Similarly, there is an isomorphism of the even-weight subcomplexes of Hochschild cochains
\[ C^*_{\text{\emph{even}}}(A^!, A^!) \cong C^*_{\text{\emph{even}}}(A^!_{\text{\emph{classical}}}, A^!_{\text{\emph{classical}}}) \]
preserving the $B_\infty$-operations \cite{GJ}.
\end{customthm}

For the symmetric algebra in particular, the remaining subcomplexes of (co)chains become acyclic after shift, up to one dimensional factors in cohomological degree 0. Consequently, much of the geometric characterization of Hochschild (co)homology of the symmetric algebra can be transferred to the exterior algebra, as we discuss in \S \ref{BV structure}:

\begin{customthm}{B}
With cohomological degree taken modulo $2$, there is an isomorphism
\[ HH_*(\Ext V, \Ext V) \cong \mathds{k} \oplus \Pi \, \Omega_{\text{\emph{odd}}} \]
where $\Omega_{\text{\emph{odd}}}$ is the space of odd-weight differential forms on $V^\vee (1)$. Similarly, there is an isomorphism of Gerstenhaber and $BV$ algebras
\[ HH^*(\Ext V, \Ext V) \cong \begin{cases}
		T_{\text{\emph{poly}}}^{\text{\emph{even}}} & \text{if $V$ is even dimensional}  \\
		T_{\text{\emph{poly}}}^{\text{\emph{even}}} \oplus \mathds{k}  & \text{if $V$ is odd dimensional}
		\end{cases} \]
where $T_{\text{\emph{poly}}}^{\text{\emph{even}}}$ is the algebra of even-weight polyvector fields. 
\end{customthm}

In the latter statement, the Gerstenhaber bracket and BV differential are extended in a natural way to the factor $\mathds{k}$ in cohomological degree 0. Because of the way Hochschild cohomology is preserved from the symmetric algebra when $V$ is even dimensional, $L_\infty$ formality of Hochschild cochains is deduced from Kontsevich's theorem in \S \ref{formality}. We show that in general this fails when $V$ is odd dimensional.

\begin{customthm}{C}
If $V$ is even dimensional, there is an $L_\infty$ quasi-isomorphism
\[ HH^*(\Ext V, \Ext V)(1) \longrightarrow C^*(\Ext V, \Ext V)(1) \]
where both sides are given their standard differential graded Lie structures. If $dim\,V = 1$, there does not exist an $L_\infty$ quasi-isomorphism.   
\end{customthm}

Regardless, in all dimensions, formal deformations of the exterior algebra are classified by formal Poisson structures in Corollary \ref{deformation theory}, just as for the symmetric algebra.

\subsection*{Acknowledgments}

The author thanks Travis Schedler, whose ideas were the start of this project and whose advising throughout was invaluable. He thanks Lee Cohn and Dan Kaplan for helpful discussions and corrections.  

\section{Preliminaries} \label{Prelim}

Throughout, $\mathds{k}$ is a field of characteristic 0 and $\otimes = \otimes_\mathds{k}$. We use the notation $V^\vee$ for $Hom_{\mathds{k}}(V, \mathds{k})$ if $V$ is a finite dimensional $\mathds{k}$-vector space, $\cong$ for isomorphism, and $\simeq$ for quasi-isomorphism. We follow the cohomological grading convention, in which differentials have degree $+1$. The term weight grading refers to the grading in which the space of generators $V$ has degree $1$ and $V^\vee$ has degree $-1$, also called Adams grading in \cite{keller}.

\subsection{Hochschild complexes}

Let $A$ be a $\mathbb{Z}$-graded $\mathds{k}$-algebra. There are two standard definitions of Hochschild (co)-homology of $A$ with different sign conventions. On the one hand, we can consider $A$ to be a differential graded algebra with vanishing differential and, letting $A^e = A \otimes A^{\text{op}}$, work in the category of differential graded $A$-bimodules, $(A^e)^{dg}\mbox{-}mod$. Objects in this category are chain complexes $(M^{*}, d)$, the $*$ indicating differential or cohomological degree, with compatible left $A^e$-action. From this perspective, the Hochschild (co)chain complex is
\[ C_*(A) : = A \otimes_{A^e} \text{Bar}^{dg}(A), \; \; \widetilde{C}^*(A) : = Hom_{A^e}(\text{Bar}^{dg}(A), A) \]
where $(\text{Bar}^{dg}(A), d_{\text{Bar}^{dg}})$ is the two-sided bar construction of $A$,
\begin{eqnarray*} 
\text{Bar}^{dg}(A) & = & \bigoplus_{n \geq 0} A \otimes A^{\otimes n} (n) \otimes A,  \\
d_{\text{Bar}^{dg}} (a_0 [a_1 | \dots | a_n] a_{n+1} ) & = & (-1)^{|a_0|} a_0 a_1 [a_2 | \dots | a_n] a_{n+1} \\
 & + & \sum_{i = 1}^{n-1} (-1)^{|a_0| + \dots + |a_i | + i} a_0 [a_1 | \dots | a_i a_{i+1} | \dots | a_n] a_{n+1} \\
& + & (-1)^{|a_0| + \dots + |a_{n-1}| + n} a_0[a_1 | \dots | a_{n-1} ] a_n a_{n+1}. 
\end{eqnarray*}
Here, $(n)$ denotes the grading shift $M(n)^r  =M^{r+n}$, and the functors $\otimes_{A^e}$ and $Hom_{A^e}$ are differential graded. If $A$ has an auxiliary $\mathbb{Z}$-grading, such as a weight grading, then $C_*(A)$ is automatically bigraded, and one can work with the subcomplex of $\widetilde{C}^*(A)$ spanned by bihomogeneous cochains. This subcomplex, which we write as $C^*(A)$, is not in general quasi-isomorphic to $\widetilde{C}^*(A)$ but is, for example, when $A$ is concentrated in cohomological degree 0 and Koszul.  

On the other hand, we can view $A$ to be concentrated in cohomological degree 0 and interpret its given $\mathbb{Z}$-grading to be distinct, which we then call the internal grading. In this case, we consider the category of chain complexes of graded $A^e$-modules, $Ch(A^e\mbox{-}gr)$. Objects are bicomplexes $(M^{*,*}, d)$ with compatible left $A^e$-actions, where the first index is cohomological degree and the second is internal degree; the differential has bidegree $(+1,0)$. The bar construction is now bigraded, and the bigraded Koszul sign rule yields the following simpler formula for the differential:
\begin{eqnarray} \label{bar d}
\text{Bar}(A) & = & \bigoplus_{n \geq 0} A \otimes A^{\otimes n} (n) \otimes A,  \nonumber \\
d_{\text{Bar}} (a_0 [a_1 | \dots | a_n] a_{n+1} ) & = &  a_0 a_1 [a_2 | \dots | a_n] a_{n+1} \nonumber \\
 & + & \sum_{i = 1}^{n-1} (-1)^{i} a_0 [a_1 | \dots | a_i a_{i+1} | \dots | a_n] a_{n+1} \nonumber \\
& + & (-1)^n a_0[a_1 | \dots | a_{n-1} ] a_n a_{n+1}. 
\end{eqnarray}
Here, $(n)$ denotes the shift $M(n)^{r, t} = M^{r+n, t}$ in cohomological degree, while we use $[n]$ to denote the shift in internal degree, $M[n]^{r,t} = M^{r, t+n}$. The Hochschild (co)chain complex is
\[ C_{*,*}(A) : = A \otimes_{A^e} \text{Bar}(A), \; \; \widetilde{C}^{*,*}(A): = Hom_{A^e}(\text{Bar}(A), A) \]
where $\otimes_{A^e}$ and $Hom_{A^e}$ are bigraded functors. This formulation appears in \cite{kassel}, in which $A$ has $\mathbb{Z} / 2 \mathbb{Z}$ internal grading. Similarly to before, if $A$ has an auxiliary $\mathbb{Z}$-grading, such as a weight grading, then one can restrict to the subcomplex of cochains spanned by trihomogeneous elements, which we write as $C^{*,*}(A)$.

As detailed in the appendix, the first version of Hochschild (co)chains is obtained from the second by totalization functors. While algebraic structures are usually written in terms of the first version, the second has the benefit of formulas that are easier to analyze under grading shifts. Hence, explicit computations are made in \S \ref{ho section} and \S \ref{coh section} in the bigraded setting and then totalized in \S \ref{BV structure} and \S \ref{formality} to study algebraic structures. 

For any differential graded, weight-graded vector space $(M^*, d_M)$, there is not only the usual homogeneous shift $(n)$ but also a shift of the differential degree by weight. Namely, let $M(\text{wt})$ be the complex whose component $M(\text{wt})^n$ consists of elements $m \in M^{n + w(m)}$ and whose differential is
\begin{equation} \label{shearing} 
d_{M(\text{wt})}(s^{-w(m)}m) : = s^{-w(m)} d_M(m) 
\end{equation}
where $s$ is suspension of differential degree and $w(a)$ is the weight of $a$. In effect, the bicomplex $M$ is sheared by the weight grading. For a chain complex of internally graded, weight-graded vector spaces $(M^{*,*}, d_M)$, we similarly define shearing of the internal degree and denote it by $[\text{wt}]$. For notational simplicity, we will omit the suspension symbol $s$ from formulas when its presence is clear from context.

\subsection{Koszul algebras}

Let $V$ be a finite dimensional $\mathds{k}$-vector space which is concentrated in internal degree $m \in \mathbb{Z}$. Everything that follows can be stated viewing $V$ instead as differential graded. The tensor algebra on $V$ over $\mathds{k}$, denoted by $T_{\mathds{k}} V$, has a weight grading that counts the length of a monomial:
\[ w(x_1 \otimes x_2 \otimes \dots \otimes x_n) = n. \]
It also has an internal grading induced from $V$:
\[  |x_1 \otimes x_2 \otimes \dots \otimes x_n| : = m w(x_1 \otimes \dots \otimes x_n). \]
We will specifically consider algebras of the form
\begin{equation} \label{form}
 A = T_{\mathds{k}} V / I
\end{equation}
where $I$ is a weight-graded and thus internally graded ideal. Such an algebra is quadratic if $I$ is generated by a subspace $R \subset V \otimes V$,
\begin{equation} \label{quadratic form}
A = T_{\mathds{k}} V / (R),
\end{equation}
and it is Koszul if the subcomplex of $\text{Bar}(A)$ given by
\begin{eqnarray*}
 K(A)_{\text{bimod}} &: = &  \bigoplus_{n \geq 0} A \otimes W_n (n) \otimes A \\
W_{0} = \mathds{k}, & & W_{n} = \bigcap_{i + 2 + j = n} V^{\otimes i} \otimes R \otimes V^{\otimes j},
\end{eqnarray*}
is quasi-isomorphic to $\text{Bar}(A)$ and hence quasi-isomorphic to $A$ in cohomological degree 0 (see, e.g., \cite{BGS}, \cite{ButKing}). The space $W = \bigoplus W_n$ is the quadratic dual coalgebra of $A$ \cite{LodayV}, but we will not have use for the coalgebra structure. Equivalently, $A$ is Koszul if the complex of left $A$-modules
\[ K(A)_{l} : = K(A)_{\text{bimod}} \otimes_A \mathds{k}, \]
where here $\mathds{k}$ denotes the augmentation module, is quasi-isomorphic to $\mathds{k}$ in cohomological degree 0. For algebras of the form \eqref{form}, we will always restrict our attention to the subcomplexes $C^{*,*}(A)$ and $C^*(A)$ rather than the full cochain complexes.

The shearing operation can be applied to the algebra $A$ in \eqref{form} viewed as a complex of vector spaces. The result $A[\text{wt}]$ (or $A(\text{wt})$ in the differential graded setting) can be endowed with the multiplication
\[ s^{-w(a)} a \cdot s^{-w(b)} b : = s^{-w(a) - w(b)} a \cdot b \; \; \forall a, b \in A, \]   
retaining the algebra structure of $A$ but altering the internal degree of elements. While at first this change seems insignificant, the opposite algebra $A^{\text{op}}$ acquires a different structure once shifted,
\begin{eqnarray*}
s^{-w(a)} a \cdot^{\text{op}} s^{-w(b)} b  & = & (-1)^{(|a| - w(a))(|b| - w(b))} s^{-w(b)} b \cdot s^{-w(a)} a \\
& = & (-1)^{(|a| - w(a))(|b| - w(b))}  s^{-w(b) - w(a)} b \cdot a. 
\end{eqnarray*}
Accordingly, the identification of a left action of $A$ on a module $M$ with a right action of $A^{\text{op}}$ and vice versa are altered.  This is the key observation in comparing the Hochschild (co)homology of $A$ to that of $A[\text{wt}]$ and, in particular, of the symmetric algebra to that of the exterior algebra. 
 
 \subsection{The symmetric and exterior algebras}
 
The symmetric algebra $\Sym V$ is defined as the free graded commutative algebra generated by $V$. In the mold of \eqref{quadratic form}, the space of relations $R$ is spanned by graded commutators,
\[ x \otimes y - (-1)^{|x| |y|} y \otimes x, \; \; x, y \in V. \]
It is well-known that $\Sym V$ is Koszul, and the quadratic dual coalgebra has components $W_n$ that are spanned by elements 
\[ \sum_{\sigma \in S_n} (sgn(\sigma))^{|V| - 1} [x_{\sigma(1)} | x_{\sigma(2)} | \dots | x_{\sigma(n)}], \; \; x_i \in V. \]
In the context of the theorem of \cite{keller} where $|V| = 0$, the Koszul dual of $\Sym V$ is $\Sym (V^{\vee}(-1))$. In terms of algebras with internal grading, we write this as $\Sym(V^{\vee}[-1])$.

The exterior algebra $\Ext V$ is defined as the free graded anticommutative algebra generated by $V$. It is constructed with $R$ in \eqref{quadratic form} spanned by graded anticommutators,
\[ x \otimes y + (-1)^{|x| |y|} y \otimes x, \; \;  x, y \in V. \]
One easily checks the equality
\begin{equation} \label{id1} \Ext V = \Sym (V[-1])[ \text{wt} ]. \end{equation}
For example, if $V$ is concentrated in even degree, the relations for $\Ext V$ and $\Sym (V[-1])$ have the same form, 
\[ x \otimes y + y \otimes x = 0, \]
but whether this represents graded commutativity or anticommutativity depends on the internal grading. To recover the anticommutativity of $\Ext V$, one must shift the degree of the generators of $\Sym (V[-1])$ by 1, their weight.  

The identity \ref{id1} implies that the bimodule/left Koszul complex of $\Ext V$ is the weight-shifted Koszul resolution of $\Sym (V[-1])$. In fact, because nothing but the internal grading changes between $K(A)$ and $K(A [\text{wt}])$, we deduce 
\begin{prop} \label{ext koszul}
If the algebra $A$ in \eqref{form} is Koszul, then $A [\text{\emph{wt}}]$ is Koszul as well.
\end{prop}

\section{Hochschild homology} \label{ho section}

Assume $A$ is of the form \eqref{form}. Under the canonical identification 
\[ C_{*,*}(A) = A \otimes_{A^e} \text{Bar}(A) \cong A \otimes_{\mathds{k}} \bigoplus_{n \geq 0} A^{\otimes n} (n), \]
the Hochschild differential $Id_{A} \otimes_{A^e} d_{\text{Bar}}$ becomes 
\begin{eqnarray} \label{Hochschild differential}
d( a[a_1 | \dots | a_n] )  & = & a a_1 [a_2 | \dots | a_n] \\
& + & \sum_{i = 1}^{n-1} (-1)^i a [a_1 | \dots | a_i a_{i + 1} | \dots | a_n] \nonumber \\
& + & (-1)^{\epsilon_n} a_n a[a_1 | \dots | a_{n-1} ] \nonumber
\end{eqnarray}
where 
\begin{eqnarray*} 
\epsilon_n  & = & n + |a_n|( |a| + |a_1| + \dots + |a_{n-1}| ) \\
& = & n + m^2 w(a_n) ( w(a) + w(a_1) + \dots + w(a_{n-1}) )
\end{eqnarray*}
on weight-homogeneous elements $a_i$. The Hochschild complex can be decomposed into even- and odd-weight subcomplexes,
\[ C_{*,*}(A) = C_{*,*}^{\text{even}}(A) \oplus C_{*,*}^{\text{odd}}(A) \]
where  $a [a_1 | \dots | a_n] \in C_{*,*}^{\text{even/odd}}(A)$ if and only if $w(a) + w(a_1) + \dots + w(a_n)$ is even/odd. We will determine how each component transforms under shearing. To start, the odd-weight subcomplex commutes with the shearing operation:

\begin{prop} \label{odd weight}
There is an isomorphism of chain complexes of bigraded vector spaces
\[ h_{*,*}: C_{*,*}^{\text{\emph{odd}}} (A[ \text{\emph{wt}} ]) \longrightarrow C_{*,*}^{\text{\emph{odd}}} (A)[ \text{\emph{wt}} ]. \]
\end{prop}

\begin{proof} 
The map $h_{*,*}: C_{*,*}(A[\text{wt}]) \to C_{*,*}(A)[\text{wt}]$ that on weight-homogeneous elements evaluates as 
\[ h_{*,*}: s^{-w(a)}a [s^{-w(a_1)} a_1 | \dots | s^{-w(a_n)} a_n] \mapsto s^{-w(a) - \dots - w(a_n)} a[a_1| \dots | a_n] \]
is clearly a levelwise isomorphism. We need to show that its restriction to the odd-weight subcomplex is compatible with differentials, the differential of the right-hand side determined by \eqref{shearing}. The compositions $h_{*,*} \circ d_{C_{*,*}(A[\text{wt}])}$ and $d_{C_{*,*}(A)[\text{wt}]} \circ h_{*,*}$ are the same up to the sign $\epsilon_n$ in \eqref{Hochschild differential}. For $A[\text{wt}]$, the factor $m$ in $\epsilon_n$ is replaced by $m-1$. In either case, if the total weight is odd, then either
\[ w(a_n) \equiv 0 \; \; \text{or} \; \; w(a) + w(a_1) + \dots + w(a_{n-1}) \equiv 0 \; \text{mod} \; 2, \]
implying $\epsilon_n \equiv n \; \text{mod} \; 2$.
\end{proof}

The even-weight subcomplex evidently changes under shearing, but the precise effect on Hochschild homology is difficult to analyze at the generality of \eqref{form}. If $A$ is Koszul, the subcomplex
\[ A \otimes_{A^e} K_{\text{bimod}}(A) \]
is quasi-isomorphic to $C_{*,*}(A)$. Writing a generic element of $W_n (n)$ as 
\begin{equation} \label{generic element}
\sum_{x_i \in V} \lambda_{x_1, \dots, x_n} [x_1 | \dots | x_n], \; \; \lambda \in \mathds{k} 
\end{equation}
we may simplify the differential on $A \otimes_{\mathds{k}} \bigoplus W_n(n)$ to
\begin{eqnarray} \label{simple diff}
d( \sum \lambda_{x_1, \dots, x_n} a [x_1 | \dots | x_n] ) & = & \sum \lambda_{x_1, \dots, x_n} a x_1 [x_2 | \dots | x_n ] \\
& + & \sum (-1)^{\epsilon_n} \lambda_{x_1, \dots, x_n} x_n a [x_1 | \dots | x_{n-1} ] \nonumber
\end{eqnarray}
where 
\begin{eqnarray*}
\epsilon_n  =  n + |x_n| (|a| + |x_1| + \dots + |x_{n-1}|) & = & n + m^2 w(x_n) ( w(a) + w(x_1) + \dots + w(x_{n-1})) \\
& = & n + m^2( w(a) + n - 1).
\end{eqnarray*}
The complex thus identified is the left Koszul complex but with a perturbed differential; for the even-weight subcomplex of $\Sym (V)[\text{wt}]$, they actually coincide.

\begin{prop} \label{evenho}
The inclusion 
\[ \mathds{k} \hookrightarrow C_{*,*}^{\text{\emph{even}}}(\Sym (V) [ \text{\emph{wt}} ])  \]
 of $\mathds{k}$ in cohomological degree 0 is a quasi-isomorphism of chain complexes of bigraded vector spaces.
\end{prop}

\begin{proof}
The weight $n$ component $W_n$ of the quadratic coalgebra dual to $\Sym (V)[\text{wt}]$ is the same as that for $\Sym V$ merely shifted in internal degree. A generic element is a linear combination of elements
\[ \sum_{\sigma \in S_n} sgn(\sigma)^{m - 1} [x_{\sigma(1)} | x_{\sigma(2)} | \dots | x_{\sigma(n)}], \; \; x_i \in V \]
where $m$ is the degree of $V$ before shearing. The differential \eqref{simple diff} for $\Sym (V)[\text{wt}]$ is
\begin{align} \label{differential} 
& d (\sum_{\sigma \in S_n} sgn(\sigma)^{m - 1} a[x_{\sigma(1)} | x_{\sigma(2)} | \dots | x_{\sigma(n)}] ) \\
& = \sum_{\sigma \in S_n} sgn(\sigma)^{m- 1} (1 + (-1)^{\epsilon_n'} ) a x_{\sigma(1)} [x_{\sigma(2)} | x_{\sigma(1)} | \dots | x_{\sigma(n)}] \nonumber
\end{align}
where 
\begin{eqnarray*}
\epsilon_n' & = & \epsilon_n + m^2w(a) + (n-1)(m - 1) \\
& = & n + (m-1)^2(w(a) + n-1) + m^2 w(a) + (n-1)(m - 1).
\end{eqnarray*}
The first term after $\epsilon_n$ comes from commuting $x_{\sigma(1)}$ and $a$, and the last term comes from the difference in signs of permutations $\sigma$ and $\sigma'$ related by
\begin{equation} \label{permutations} \sigma(1) = \sigma'(n),\; \sigma(2) = \sigma'(1), \dots, \sigma(n) = \sigma'(n-1).
\end{equation}  
Restricting $d$ to the even-weight subcomplex, i.e., $w(a) + n \equiv 0 \; \text{mod} \; 2$, observe
\[ \epsilon_n'  \equiv n + m-1 + m w(a) + (n-1)(m-1) \equiv m(w(a) + n) \equiv 0 \; \text{mod} \; 2, \]
so $d$ is 2 times the left Koszul differential. This complex is quasi-isomorphic to $\mathds{k}$ in cohomological degree 0 (Proposition \ref{ext koszul}). 
\end{proof}

With identity \ref{id1}, we have proven

\begin{thm} \label{ho}
There is a quasi-isomorphism of chain complexes of bigraded vector spaces
\[ \mathds{k} \oplus C_{*,*}^{\text{\emph{odd}}}(\Sym (V[-1]))[ \text{\emph{wt}} ] \simeq C_{*,*}(\Ext V). \]
If the internal degree is taken modulo $2$, the quasi-isomorphism is
\[\mathds{k} \oplus \Pi \, C_{*,*}^{\text{\emph{odd}}}(\Sym (\Pi \,V)) \simeq C_{*,*}(\Ext V) \]
where $\Pi$ is the parity inversion operator.
\end{thm}

\section{Hochschild cohomology} \label{coh section}

Again assume $A$ is of the form \eqref{form}. Our computation of Hochschild cohomology parallels the previous section. The Hochschild codifferential is
\[ d(f) : = (-1)^{n} f \circ d_{\text{Bar}}, \; \; f \in Hom_{A^e}(A \otimes A^{\otimes n} (n) \otimes A, A). \] Under the canonical identification
\[ Hom_{A^e} (\text{Bar}(A), A) \cong Hom_{\mathds{k}}( \bigoplus_{n \geq 0} A^{\otimes n} (n), A ), \]
the codifferential evaluates on a weight-homogeneous cochain $f$ as
\begin{eqnarray} \label{codiff} 
d(f)([a_1 | a_2 | \dots | a_n]) & = &  (-1)^{\delta_n} a_1 f([a_2 | \dots | a_n]) \\
& + & \sum_{i = 1}^{n-1} (-1)^{i+n} f([a_1| \dots | a_i a_{i+1} | \dots | a_n]) \nonumber \\
& + & f([a_1| \dots | a_{n-1}]) a_n \nonumber
\end{eqnarray}
where $\delta_n = n + |a_1| |f| =n + m^2 w(a_1) w(f)$. The Hochschild cochain complex can be decomposed into even- and odd-weight subcomplexes,
\[ C^{*,*}(A) = C^{*,*}_{\text{even}}(A) \oplus C^{*,*}_{\text{odd}}(A) \]
where $f \in C^{*,*}_{\text{even}/\text{odd}} (A)$ if and only if 
\[ w(f) = w( f([a_1 | \dots | a_n]) ) - w([a_1 | \dots | a_n]) \; \; \text{is even/odd}. \]
This time, the even-weight subcomplex commutes with the shearing operation:

\begin{prop} \label{even weight}
There is an isomorphism of chain complexes of bigraded vector spaces
\[  h^{*,*}: C^{*,*}_{\text{\emph{even}}}(A[ \text{\emph{wt}} ]) \longrightarrow C^{*,*}_{\text{\emph{even}}}(A)[\text{\emph{wt}} ]. \] 
\end{prop}

\begin{proof}
A weight-homogeneous map $f \in Hom_{\mathds{k}}((A[\text{wt}])^{\otimes n} (n), A[\text{wt}])$ 
determines an element $\bar{f} \in Hom_{\mathds{k}}(A^{\otimes n}(n), A)$ evaluating as
\[ \bar{f} ([a_1 | \dots | a_n]) = s^{w(f) + w(a_1) + \dots + w(a_n)} f([ s^{-w(a_1)} a_1 | \dots | s^{-w(a_n)} a_n]). \]
The map $h^{*,*}: C^{*,*}(A[\text{wt}]) \to C^{*,*}(A)[\text{wt}]$ defined by
\[ h^{*,*}: f \mapsto s^{-w(f)} \bar{f} \]
is clearly a levelwise isomorphism. We need to show that its restriction to the even-weight subcomplex is compatible with differentials, the differential of the right-hand side determined by \eqref{shearing}. The compositions $h^{*,*} \circ d_{C^{*,*}(A[\text{wt}])}$ and $d_{C^{*,*}(A)[\text{wt}]} \circ h^{*,*}$ are the same up to the sign $\delta_n$ in \eqref{codiff}. For $A[\text{wt}]$, the factor $m$ in $\delta_n$ is replaced by $m-1$. In either case, if $w(f) = w(\bar{f})$ is even, then $\delta_n \equiv n \; \text{mod} \; 2$.  
\end{proof}

Analogous to the situation of Hochschild homology, the odd-weight subcomplex changes under shearing, but the effect on cohomology is difficult to measure in general. If $A$ is Koszul, the cochain complex
\[ Hom_{A^e}( K_{\text{bimod}}(A), A) \cong Hom_{\mathds{k}} \big( \bigoplus_{n \geq 0} W_n (n), A \big) \]
is quasi-isomorphic to $C^{*,*}(A)$. In the notation \eqref{generic element}, the codifferential simplifies to 
\begin{eqnarray} \label{koszul codiff}
df \big( \sum \lambda_{x_1, \dots, x_n} [x_1 | \dots | x_n] \big) & = & \sum (-1)^{\delta_n} \lambda_{x_1, \dots, x_n} x_1 f( [x_2 | \dots | x_n ]) \\
& + & \sum \lambda_{x_1, \dots, x_n} f([x_1 | \dots | x_{n-1} ]) x_n \nonumber
\end{eqnarray}
where $\delta_n = n + |x_1| |f| = n + m^2 w(f)$. This complex is reminiscent of the left Koszul complex but with a perturbed differential. Indeed, for $A = \Sym (V)[\text{wt}]$, the odd-weight subcomplex is precisely the odd-weight part of the Koszul complex. Let $\text{det}(V)$ be the determinant line, i.e., the highest weight component of $\Ext V$ if $|V|$ is even or of $\Sym V$ if $|V|$ is odd. 

\begin{prop} \label{oddcoh}
Let $N = \text{dim} \, V$ and $m = |V|$. 
\begin{enumerate}
\item If $N$ is even, $C^{*,*}_{\text{\emph{odd}}} (\Sym (V) [\text{\emph{wt}}])$ is acyclic.
\item If $N$ and $m$ are odd, the inclusion
\[ \text{det}(V)[N] \longhookrightarrow C^{*,*}_{\text{\emph{odd}}} (\Sym (V) [\text{\emph{wt}}] ) \]
is a quasi-isomorphism.
\item If $N$ is odd and $m$ is even, there is a quasi-isomorphism 
\[ \text{det}(V[1])^\vee(-N)  \simeq C^{*,*}_{\text{\emph{odd}}} (\Sym (V)[\text{\emph{wt}}]). \]
\end{enumerate} 
\end{prop}

\begin{proof}
Recall that a generic element of $W_n$ for $\Sym(V)[\text{wt}]$ is a linear combination of elements 
\[ \sum_{\sigma \in S_n} sgn(\sigma)^{m - 1} [x_{\sigma(1)} | x_{\sigma(2)} | \dots | x_{\sigma(n)}], \; \; x_i \in V. \]
The codifferential \eqref{koszul codiff} evaluates on such elements as
\begin{align*}
& df \big( \sum_{\sigma \in S_n} sgn(\sigma)^{m - 1} [x_{\sigma(1)} | \dots | x_{\sigma(n)}]  \big) \\
& = \sum_{\sigma \in S_n} sgn(\sigma)^{m - 1} ( (-1)^{\delta_n'} + 1) f([x_{\sigma(1)} | \dots | x_{ \sigma(n-1)} ] ) x_{\sigma(n)}
\end{align*}
where
\begin{eqnarray*}
\delta_n' & = & \delta_n + (n-1)(m - 1) + m^2 (n-1 + w(f)) \\
& = & n + (m-1)^2 w(f) + (n-1)(m-1) + m^2 (n-1 + w(f)).
\end{eqnarray*}
The first term after $\delta_n$ comes from the difference in signs of two permutations $\sigma$ and $\sigma'$ as in \eqref{permutations}, and the last term comes from commuting $x_{\sigma(n)}$ and $f( [x_{\sigma(1)} | \dots | x_{\sigma(n-1)}] )$. If $w(f)$ is odd, then 
\begin{align*}
& df \big( \sum_{\sigma \in S_n} sgn(\sigma)^{m - 1} [x_{\sigma(1)} | \dots | x_{\sigma(n)}]  \big) \\
& = 2 \sum_{\sigma \in S_n} sgn(\sigma)^{m - 1} f( [x_{\sigma(1)} | \dots | x_{ \sigma(n-1)} ] ) x_{\sigma(n)}.
\end{align*}
This formula can be written more suggestively as follows. The component $W_n$ is isomorphic as a vector space to the $n^{\text{th}}$-symmetric power of $V[1]$,
\begin{equation} \label{koszul identification}
\sum_{\sigma \in S_n} sgn(\sigma)^{m-1} [x_{\sigma(1)} | \dots | x_{\sigma(n)} ] \in W_n \longleftrightarrow x_1 \dots x_n \in \Sym^n (V[1]).
\end{equation}
In these terms, the codifferential is
\begin{eqnarray} \label{suggestive}
df( x_1 \dots x_n) & = & 2 \sum_{i =1}^n (-1)^{(m-1)(n-i)} f(x_1 \dots \hat{x}_i \dots x_n )x_i \\
& = &  (-1)^{(m-1)(n-1)} 2 \sum_{i =1}^n f \big( \frac{\partial}{\partial x_i} \big(x_1 \dots x_i \dots x_n \big)\big) x_i . \nonumber
\end{eqnarray}

Since $\Sym^n(V[1])$ is finite dimensional, we have the canonical isomorphism 
\begin{equation} \label{hom identity} Hom_{\mathds{k}} (\Sym^n (V[1])(n), \Sym (V)[ \text{wt} ]) \cong \Sym (V)[ \text{wt} ] \otimes \Sym^n(V[1])^\vee (-n).
\end{equation}
If $m$ is even, then $\Sym (V[1])$ is a graded Frobenius algebra: we have the identification
\begin{equation} \label{Frob duality}
\Sym^n (V[1])^\vee \cong \Sym^{N-n}(V[1])[N(m-1)] 
\end{equation}
depending on a choice of an element in $\text{det}(V[1]) =\Sym^N(V[1])$. With this isomorphism,
\begin{equation} \label{frob}
Hom_{\mathds{k}} (\Sym^n (V[1])(n), \Sym (V)[ \text{wt} ]) \cong \Sym (V)[ \text{wt} ] \otimes \Sym^{N-n}(V[1])[N(m-1)](-n).
\end{equation}
The graded vector space underlying this complex is the same as that underlying the shifted left Koszul resolution $K(\Sym(V)[\text{wt}])_l [N(m-1)](-N)$ written in terms of \eqref{koszul identification}. Notice, however, that the subcomplex on the right-hand side corresponding to the odd-weight subcomplex on the left depends on $N$: if $N$ is even, it is the odd-weight subcomplex, but if $N$ is odd, the parity of weight is reversed. In either case, it is clear that, pushed through these isomorphisms, the codifferential \eqref{suggestive} becomes
\[ d(p \otimes x_1 \dots x_{N-n}) = (-1)^{(m-1)(n-1)} 2 \sum_{i = 1}^{N-n} p x_i \otimes \frac{\partial}{\partial x_i}\big(x_1 \dots x_i \dots x_{N-n}\big), \]
precisely $2$ times the left Koszul differential. The odd-weight part of the shifted left Koszul resolution is acyclic, while the even-weight part is quasi-isomorphic to $\mathds{k}[N(m-1)](-N)$. Reversing \eqref{Frob duality}, we identify this one-dimensional factor with the dual of the shifted determinant line, $\Sym^N(V[1])^\vee(-N) = \text{det} (V[1])^\vee(-N) $, in the original complex \eqref{hom identity}.

If instead $m$ is odd, then $\Sym (V)[\text{wt}]$ is graded Frobenius, and similarly to before, it is isomorphic to its dual $\Sym(V)[\text{wt}]^\vee [N(1-m)]$:
\begin{align} \label{odd duality}
Hom_{\mathds{k}} (\Sym^n (V[1])(n), \Sym (V)[ \text{wt} ]) \\ 
\cong \Sym (V)[\text{wt}]^\vee [N(1-m)] & \otimes \Sym^n (V[1])^\vee (-n) \nonumber \\
\cong \Big( \Sym (V) [\text{wt}] [N(m-1)]  & \otimes \Sym^n (V[1]) (n) \Big)^\vee. \nonumber
\end{align}
The graded vector space underlying this complex is the same as that underlying
\[ Hom_{\mathds{k}}( K(\Sym(V)[\text{wt}])_l [N(m-1)], \mathds{k} ). \] 
Once again, the parity of weight is preserved by the first isomorphism in \eqref{odd duality} if $N$ is even and reversed if $N$ is odd. In either case, the codifferential \eqref{suggestive} through these identifications is the dual of
\[ d(p \otimes x_1 \dots x_n) = (-1)^{(m-1)(n-1)} 2 \sum_{i = 1}^{n} p x_i \otimes \frac{\partial}{\partial x_i}\big(x_1 \dots x_i \dots x_{n}\big), \] 
the left Koszul differential. Since the functor $Hom_{\mathds{k}}( \cdot, \mathds{k})$ is exact, the cohomology of $K(\Sym(V)[\text{wt}])_l[N(m-1)]$ is simply the linear dual of the cohomology of \eqref{odd duality}. We thus trace $\mathds{k}[N(m-1)]$ back to the shifted determinant line $\text{det}(V)[N]$.
\end{proof}

With identity \ref{id1}, we have proven

\begin{thm} \label{coh} Let $N = \text{dim} \, V$ and $m = |V|$. There are quasi-isomorphisms of chain complexes of bigraded vector spaces
\[ C^{*,*}_{\text{\emph{even}}}(\Sym (V[-1]))[ \text{\emph{wt}} ] \simeq C^{*,*} (\Ext V) \]
if $N$ is even;
\[  C^{*,*}_{\text{\emph{even}}}(\Sym (V[-1]))[ \text{\emph{wt}}] \oplus \mathds{k}[-Nm]  \simeq C^{*,*} (\Ext V) \]
if $N$ is odd and $m$ is even; and
\[  C^{*,*}_{\text{\emph{even}}}(\Sym V[-1])[ \text{\emph{wt}}] \oplus \mathds{k}[Nm](-N) \simeq  C^{*,*} (\Ext V) \]
if $N$ and $m$ are odd. If the the internal degree is taken modulo $2$, the shifts $[\text{\emph{wt}}]$ are trivial.
\end{thm}

\section{Algebraic structure on Hochschild cohomology} \label{BV structure}

Continuing with the assumption that $A$ is of the form \eqref{form}, suppose $|V| = 1$, so the internal grading of $A$ coincides with weight and $A[\text{wt}]$ is concentrated in degree 0. Thus far, we have worked with the Hochschild (co)chain complexes bigraded by cohomological and internal degrees. We now totalize the complexes to study algebraic structures, for which $A$ should be considered a differential graded algebra. This is achieved explicitly by the isomorphisms $F$ and $G$ in \eqref{iso 1} and \eqref{iso 3}. Combining them with the maps in Propositions \ref{odd weight} and \ref{even weight}, we have isomorphisms of differential graded vector spaces 
\begin{eqnarray*}
h_*: = F \circ \Tplus (h_{*,*}) \circ F^{-1} : C_*^{\text{odd}}(A(\text{wt})) \longrightarrow C_*^{\text{odd}}(A)(\text{wt}), \\
h^*: = G \circ \Tplus(h^{*,*}) \circ G^{-1}: C_{\text{even}}^* (A(\text{wt})) \longrightarrow C_{\text{even}}^*(A) (\text{wt}).
\end{eqnarray*}
Suppressing the suspension symbol, we see in particular that $h^*$ maps a weight-homogeneous $f \in C^n_{\text{even}}(A(\text{wt}))$ to the element in $C_{\text{even}}(A)(\text{wt})^n$ defined by
\begin{equation} \label{h evaluation} 
h^*(f)([a_1 | \dots | a_n]) = (-1)^{(n-1)w(a_1) + (n-2)w(a_2) + \dots + w(a_{n-1})} f([a_1 | \dots | a_n]).
\end{equation}
It is understood that the $a_i$ on the left are viewed as elements of $A$ and those on the right are viewed as elements of $A(\text{wt})$. We will show that $h^*$ preserves the algebraic structure on Hochschild cochains, including the usual operations descending to cohomology.

\subsection{Gerstenhaber structure} \label{Gers structure}

Recall that for an arbitrary differential graded algebra $A$, the cup product of two Hochschild cochains  $f \in C^{n}(A)$ and $g \in C^{p}(A)$ is defined by
\[ f \cup g ([a_1 | \dots | a_{n + p}]) = (-1)^{ p(|a_1| + \dots + |a_n| + n)} f( [a_1 | \dots | a_n]) g([a_{n+1} | \dots | a_{n + p}]). \] 
The $k^{\text{th}}$ brace operation on cochains $f_i \in C^{n_i}(A)$ is defined as
\[ f_0 \{f_1, \dots, f_k \} : = \sum_{1 \leq j_1 < \dots < j_k \leq  n_0} f_0 \circ_{j_1, \dots, j_k} (f_1 , \dots, f_k) \] 
where $\circ_{j_1, \dots, j_k }$ denotes inserting $f_i$ into slot $j_i$ of $f_0$. Together, the cup product and brace operations satisfy certain higher homotopical identities and give an action of the operad $B_\infty$ on $C^*(A)$ \cite{GJ}. The induced Lie bracket on the shifted complex $C^*(A)(1)$,
\[ [ f_0, f_1] := f_0 \{f_1\} - (-1)^{(n_0 - 1)(n_1 - 1)} f_1 \{f_0 \}, \]
is the Gerstenhaber bracket, originally studied in \cite{Gers}. 

For $A$ in our desired form \eqref{form}, these operations evidently preserve the weight grading and so restrict to $C^*_{\text{even}}(A)$. The sheared complex $C^*_{\text{even}}(A)(\text{wt})$ retains the operations from the unshifted complex,  
\begin{eqnarray*}
s^{-w(f)} f \cup s^{-w(g)} g & := & s^{-w(f) - w(g)} f \cup g  \\
s^{-w(f_0)} f_0 \{ s^{-w(f_1)}f_1, \dots, s^{-w(f_n)} f_n\} & := & s^{-w(f_0) - \dots - w(f_n)} f_0 \{ f_1, \dots, f_n \}.
\end{eqnarray*} 

\begin{prop} \label{B infinity}
The map $h^*$ is an isomorphism of $B_\infty$-algebras.
\end{prop}

\begin{proof}
Let $f \in C^n_{\text{even}}(A(\text{wt}))$ and $g \in C^p_{\text{even}}(A(\text{wt}))$. Observe
\begin{eqnarray*} 
h^*(f) \cup h^*(g) ([a_1 | \dots | a_{n+p}]) & = & (-1)^{\kappa_\cup} h^*(f) ([a_1 | \dots | a_n]) h^*(g) ([a_{n+1} | \dots | a_{n+p} ]) \\
& = & (-1)^{\kappa_\cup + \kappa_m + \kappa_n} f([a_1 | \dots | a_n]) g([a_{n+1} | \dots | a_{n + p}]) 
\end{eqnarray*}
where
\begin{eqnarray*}
\kappa_\cup & = & (w(a_1) + \dots + w(a_n) + n)(p +w(g)) \equiv ( w(a_1) + \dots + w(a_n) + n)p  \; \; \text{mod} \; 2,  \\
\kappa_n & = & (n-1) w(a_1) + (n-2) w(a_2) + \dots + w(a_{n-1}), \\
\kappa_p & = & (p-1) w(a_{n+1})+ (p-2)w(a_{n+2}) + \dots + w(a_{n + p-1}).
\end{eqnarray*}
But this precisely equals
\begin{align*} 
h^*(f \cup g)([a_1 | \dots | a_{n + p}]) & =  (-1)^{(n +p - 1)w(a_1) + \dots + w(a_{n + p-1})} f \cup g ([a_1 | \dots | a_{n + p}]) \\
& = (-1)^{np + (n +p - 1)w(a_1) + \dots + w(a_{n + p-1})} f([a_1 | \dots | a_n]) g([a_{n+1} | \dots | a_{n + p} ]).
\end{align*}
Hence, $h^*$ preserves the cup product. Similar computation shows $h^*$ respects all the brace operations.
\end{proof}

\begin{cor} \label{partial calculus}
The map $h^*$ descends to an isomorphism of Gerstenhaber algebras
\[ HH^*_{\text{\emph{even}}} (A(\text{\emph{wt}}) ) \longrightarrow HH^*_{\text{\emph{even}}} (A) (\text{\emph{wt}}). \]
\end{cor}

Having related Hochschild (co)homology of the exterior algebra to that of the symmetric algebra in Theorems \ref{ho} and \ref{coh}, we can provide a geometric interpretation through the Hochschild-Kostant-Rosenberg (HKR) formalism \cite{HKR}. Within the Hochschild cochain complex of $A = \Sym V$ where $|V| = 0$ or $1$ is the subcomplex $D_{\text{poly}} = D_{\text{poly}}(V^\vee)$ consisting of polydifferential operators, namely, the span of tensor products of algebraic differential operators on $A$. The cohomology of $D_{\text{poly}}$ is the space of algebraic polyvector fields, $T_{\text{poly}}  = T_{\text{poly}}(V^\vee) : = \Sym_{A} Der(A)$ where $Der(A)$ is the space of $\mathds{k}$-linear derivations of $A$. With symmetric product $\bullet$ and Schouten-Nijenhuis bracket $[\cdot, \cdot]_{\text{SN}}$, polyvector fields form a Gerstenhaber algebra. In fact, the composition
\begin{equation} \label{HKR map} 
\text{HKR}: T_{\text{poly}} \to D_{\text{poly}} \hookrightarrow C^*(A) , \; \ \gamma_1 \bullet \dots \bullet \gamma_n \mapsto \frac{1}{n!} \sum_{\sigma \in S_n} \epsilon_{\sigma} \gamma_{\sigma(1)} \otimes \dots \otimes \gamma_{\sigma(n)},
\end{equation}
where $\epsilon_\sigma$ is determined by Koszul sign rule, is a quasi-isomorphism of differential graded vector spaces, and the induced isomorphism on cohomology $T_{\text{poly}} \to HH^*(A)$ is one of Gerstenhaber algebras.

Dually, the HKR theorem identifies the Hochschild homology of the symmetric algebra as algebraic differential forms. Specifically, if $\Omega_K$ denotes the space of K\"{a}hler forms of $A$, then the map
\[ \text{HKR}: C_*(A) \to \Omega(V^\vee) : = \Sym_{A} \Omega_{K}, \; \; a_0 [a_1 | \dots | a_n] \mapsto \frac{1}{n!} a_0 da_1 \dots da_n, \]
induces an isomorphism $HH_*(A) \to \Omega$. Here, $\Omega$ is viewed as having trivial differential.

The spaces $D_{\text{poly}}$, $T_{\text{poly}}$, and $\Omega$ are naturally weight-graded and so decompose into even- and odd-weight subspaces. 

\begin{thm} \label{calculus} Suppose $|V| = 0$ and let $N = \text{dim} \, V$.
\begin{enumerate}
\item There is an isomorphism of differential graded vector spaces
\[  HH_*(\Ext V) \longrightarrow \mathds{k} \oplus \Omega_{\text{\emph{odd}}} (V^\vee(1))(\text{\emph{wt}}). \]
\item If $N$ is even, there is an isomorphism of Gerstenhaber algebras
\[ \Big(T_{\emph{poly}}^{\emph{even}}(V^\vee(1))(\text{\emph{wt}}), \, \bullet \, , [\cdot, \cdot]_{\text{\emph{SN}}} \Big) \longrightarrow \Big( HH^*(\Ext V), \cup, [\cdot, \cdot] \Big). \]
\item If $N$ is odd, there is an isomorphism of Gerstenhaber algebras 
\[  \Big( T_{\emph{poly}}^{\emph{even}}(V^\vee(1))(\text{\emph{wt}}) \oplus \Ext^N V,  \, \bullet \, , [\cdot, \cdot]_{\text{\emph{SN}}} \Big) \longrightarrow \Big( HH^*(\Ext V), \cup, [\cdot, \cdot] \Big) \]
where on the left the symmetric product and Schouten-Nijenhuis bracket are extended according to the rules
\begin{eqnarray*} \label{cup} \forall \, \gamma \in \Ext^N V, & & \alpha \bullet \gamma =  
	\begin{cases}
		\alpha \gamma & \text{if} \; \alpha \in \mathds{k} \subset T_{\text{\emph{poly}}}^{\text{\emph{even}}}(\text{\emph{wt}})^0 \\
		0 & \text{otherwise}
	\end{cases} \\
 & & [\alpha, \gamma ]_{\text{\emph{SN}}} =
	\begin{cases}
		(h^*)^{-1}\alpha \, (\gamma) & \text{if} \; \alpha \in T_{\text{\emph{poly}}}^{\text{\emph{even}}}(\text{\emph{wt}})^1 \\
		0 & \text{otherwise}. 
	\end{cases} \end{eqnarray*}
\end{enumerate}
\end{thm}

\begin{proof}
The first statement is a result of Theorem \ref{ho}, the isomorphism $h_*$, and the HKR isomorphism. The second statement follows from Corollary \ref{partial calculus} and the HKR isomorphism. This also establishes the isomorphism of the even-weight subspaces in the third statement. The remaining assertion in the third statement follows from comparing cohomological degree and weight. If $\alpha \in HH^n_{\text{even}}(\Ext V)$, then $\alpha \cup \gamma \in HH^n_{\text{odd}}(\Ext V)$, which is trivial unless $n = 0$ and $w(\alpha) = 0$. Also $[\alpha, \gamma] \in HH^{n-1}_{\text{odd}}(\Ext V)$, which is trivial unless $n = 1$.
\end{proof}

This statement in particular recovers the results regarding the additive and cup product structure in \cite{BGMS}, \cite{BE}, \cite{Oppermann}, \cite{HanXu} and \cite{LamVav}, as well as the results about the Gerstenhaber structure in \cite{GNW} and \cite{Grimley}.

\begin{rem} \label{Hochschild koszul duality}
As promised by Keller's theorem \cite{keller}, the Koszul dual algebras $\Sym(V(-1))$ and $\Sym (V^\vee)$ have isomorphic Hochschild cohomology as Gerstenhaber algebras. The isomorphism is easily written in coordinates.  If $\{x_1^\vee, \dots, x_n^\vee \}$ is a basis of $V^\vee$ and $\{ \xi_1, \dots, \xi_n \}$ is the shifted dual basis in $V(-1)$, it is
\begin{equation} \label{koszul interchange}
T_{\text{poly}}(V^\vee (1)) \longleftrightarrow T_{\text{poly}}(V), \; \; \xi_{j_1} \dots \xi_{j_t} \frac{\partial}{\partial \xi_{i_1}} \dots \frac{\partial}{\partial \xi_{i_s}} \longleftrightarrow x_{i_1}^\vee \dots x_{i_s}^\vee \frac{\partial}{\partial x_{j_1}^\vee} \dots \frac{\partial}{\partial x_{j_t}^\vee}. 
\end{equation}
That is, Koszul duality interchanges the ``polynomial" and ``polyvector" components of Hochschild cohomology. We may compose this isomorphism with the isomorphisms in statements $(2)$ and $(3)$ of Theorem \ref{calculus} to relate the cohomology of $\Ext V$ to $T_{\text{poly}}^{\text{even}}(V)$, polyvector fields on ungraded affine space.
\end{rem}

The theorem provides an explicit basis for $HH^*(\Ext V)$ that is convenient for algebraic computations. Letting $\{x_1, \dots, x_n\}$ be a basis for $V$, we may write the cohomological basis as the set of all polyvector fields of the form
\[ x_{j_1} \dots x_{j_t}  \frac{\partial}{\partial x_{i_1}} \dots \frac{\partial}{\partial x_{i_s}}, \;  \; 1 \leq j_1 < \dots < j_t \leq n, \; \; 1 \leq i_1 \leq \dots \leq i_s \leq n,  \; \; s + t \equiv 0 \; \text{mod} \; 2 \] 
corresponding to 
$\xi_{j_1} \dots \xi_{j_t}  \frac{\partial}{\partial \xi_{i_1}} \dots \frac{\partial}{\partial \xi_{i_s}} \in T_{\text{poly}}(V^\vee (1))$, augmented by a chosen determinant form in $\Ext^N V$ when $N$ is odd.  Under the identification $h^*$ in \eqref{h evaluation} and the HKR map, these basis elements are represented by cocycles in $C^*(\Ext V)$ evaluating as
\begin{eqnarray} \label{cocycle evaluation}
x_{j_1} \dots x_{j_t} \frac{1}{n!} \sum_{\sigma \in S_n} \frac{\partial}{\partial x_{i_{\sigma(1)}}} \otimes \dots \otimes \frac{\partial}{\partial x_{i_{\sigma(s)}}}([a_1 | \dots | a_n]) \\
= (-1)^{(n-1) w(a_1) + \dots + w(a_{n-1})} x_{j_1} \dots x_{j_t} \frac{1}{n!} \sum_{\sigma \in S_n} \frac{\partial}{\partial x_{i_{\sigma(1)}}}(a_1) \dots  \frac{\partial}{\partial x_{i_{\sigma(s)}}}(a_n) \nonumber
\end{eqnarray}
on weight-homogeneous elements $a_i$ of $\Ext V$.
\subsection{BV structure}

As is well known, the space of polyvector fields is a Batalin-Vilkovisky (BV) algebra. Specifically, there is a square-zero, cohomological degree $-1$ operator $\Delta$ recovering the Gerstenhaber bracket by the identity
\[ [f, g] = \Delta(fg) - \Delta(f)g - (-1)^{|f|} f \Delta(g). \]
For $T_{\text{poly}}(V^\vee (1))$, the BV operator is the divergence operator with respect to a chosen determinant form $\omega \in \Sym (V(-1))$. If in coordinates $\omega = \xi_1 \xi_2 \dots \xi_N$, then
\begin{equation} \label{divergence}
\text{div}_\omega = \sum_{i = 1}^N \frac{\partial}{\partial \xi_i} \frac{\partial}{\partial \theta_i}, \; \; \; \theta_i = \frac{\partial}{\partial \xi_i}. 
\end{equation}
The BV operator for $T_{\text{poly}}(V)$ has the same form under the isomorphism \eqref{koszul interchange}.

Recent work by Volkov \cite{Volkov} and, independently, by Lambre et al. \cite{LZZ} demonstrates the existence of BV structures on the Hochschild cohomology of certain Frobenius algebras. Recall that a Frobenius algebra $A$ has a nondegenerate bilinear form $\langle \cdot, \cdot \rangle : A \otimes A \to \mathds{k}$ with a distinguished algebra automorphism $\nu$, called the Nakayama automorphism, determined by
\[ \langle a, \cdot \rangle = \langle \cdot, \nu(a) \rangle \; \; \forall \, a \in A. \]  
The above-mentioned results specifically concern Frobenius algebras with semisimple Nakayama automorphism. Volkov generalizes the BV structure in \cite{Tradler} for algebras with symmetric, nondegenerate inner products, which is conjectured to coincide with BV structures arising from string topology \cite{ChasSul}. The authors of \cite{LZZ} show that the Connes operator on Hochshild homology with coefficients twisted by the Nakayama automorphism dualizes to a BV operator. The two BV structures are expected to agree generally.

The exterior algebra $\Ext V$ is a basic example of a Frobenius algebra with semisimple Nakayama automorphism. When $\text{dim} \, V$ is odd, the automorphism $\nu$ is trivial, but when $\text{dim} \, V$ is even, it detects the parity of weight,
\[ \nu (a) = (-1)^{w(a)} a. \]
Since $\text{div}_\omega$ is weight-homogeneous, the identification in Theorem \ref{calculus} immediately imparts a BV structure to $HH^*_{\text{even}}(\Ext V)$ recovering the Gerstenhaber structure. When $N$ is odd, the operator $\text{div}_\omega$, having cohomological degree $-1$, is extended by zero to $\Ext^N V$. It is straightforward to check that the resulting BV structure on the entire Hochschild cohomology $HH^*(\Ext V)$ is the one described by general formulas in \cite{Volkov} and \cite{LZZ}. 

\begin{thm} \label{BV}
Suppose $|V| = 0$. Let $N = \text{dim} \, V$, and let $\text{\emph{div}}_{\omega}$ be the divergence operator associated to a determinant form $\omega \in \Sym (V(-1))$ restricted to $T_{\text{\emph{poly}}}^{\text{\emph{even}}}(V^\vee (1))$. 
\begin{enumerate}
\item If $N$ is even, $HH^*(\Ext V)$ is isomorphic as a BV algebra to
\[ \Big(T_{\emph{poly}}^{\emph{even}}(V^\vee(1))(\text{\emph{wt}}), \, \bullet \, , \text{\emph{div}}_{\omega}  \Big). \]
\item If $N$ is odd, $HH^*(\Ext V)$ is isomorphic as a BV algebra to
\[ \Big( T_{\emph{poly}}^{\emph{even}}(V^\vee(1))(\text{\emph{wt}}) \oplus \Ext^N V, \, \bullet \, , \text{\emph{div}}_{\omega} \Big) \]
where $\text{\emph{div}}_{\omega}$ is extended by 0 to $\Ext^N V$.
\end{enumerate}
In either case, the induced Gerstenhaber structure is the one of Theorem \ref{calculus}.
\end{thm}

Thus, when $N$ is even, the exterior algebra provides a simple example of a BV structure coming from a duality with twisted Hochschild homology (\cite{LZZ}). The result here was independently discovered by Weiguo Lv in a forthcoming paper, but his methods are along the lines of \cite{HanXu} and are very different.

\section{Formality theorems and deformation theory} \label{formality}

Assume for now that $\mathds{k} = \mathbb{R}$ and $|V| = 0$. In broad terms, the Hochschild cochain complex $C^*( \mathcal{O}(\mathbb{R}^n))$ of algebraic functions has the special property of being algebraically formal. Kontsevich \cite{Kont} proved that the HKR map \eqref{HKR map} extends to an $L_\infty$ quasi-isomorphism between the shifted complexes $HH^*(\mathcal{O}(\mathbb{R}^n))(1)$ and $C^*(\mathcal{O}(\mathbb{R}^n))(1)$ given their standard differential graded Lie structures. Tamarkin \cite{Tamarkin} exhibited a functorial $G_\infty$ structure on $C^*(\mathcal{O}(\mathbb{R}^n))$ and the existence of a $G_\infty$ quasi-isomorphism extending the HKR map. Recently, Willwacher \cite{Willwacher} and Campos \cite{Campos} proved versions of formality involving homotopy braces and homotopy BV structures.

Catteneo and Felder \cite{CF} adapted Kontsevich's $L_\infty$ formality morphism to the setting of supermanifolds, a result stated implicitly in the original paper \cite{Kont}. As for the ungraded case, the morphism consists of a sequence of linear maps
\[ U_n: T_{\text{poly}}(V^\vee(1))(1)^{\otimes n} \longrightarrow D_{\text{poly}}(V^\vee(1))(1) \] 
satisfying the $L_\infty$ homotopical identities. The first component $U_1$ is the HKR map \eqref{HKR map}, and the higher components $U_n$ are weighted summations over Kontsevich graphs $\Gamma$ \cite{Kont} of maps $D_\Gamma$. In coordinates $\xi_1, \dots, \xi_N$, for $V(-1)$, the $D_\Gamma$ are defined by
\begin{equation} \label{CF formality} D_\Gamma(\gamma_1, \dots, \gamma_n)(f_1, \dots, f_m) = \pi \mu \Big( \prod_{(i,j) \in \Gamma} \sum_{ k = 1}^N \frac{\partial}{\partial \theta_k^{(i)}} \otimes \frac{\partial}{\partial \xi_k^{(j)}}\Big) (\gamma_1 \otimes \dots \otimes \gamma_n \otimes f_1 \otimes \dots \otimes f_m )  
\end{equation}
where 
\begin{itemize}
\item $\Gamma$ has $n$ vertices of type I and $m$ vertices of type II;
\item $\theta_k = \frac{ \partial}{\partial \xi_k}$, $\gamma_p \in T_{\text{poly}}(V^\vee(1))$ for all $1 \leq p \leq n$, and $f_t \in \Sym (V(-1))$ for all $1 \leq t \leq m$;
\item $(i)$ and $(j)$ mean the the operators $\frac{\partial}{\partial \theta_k^{(i)}}$ and $\frac{\partial}{\partial \xi_k^{(j)}}$ are applied to the $i^{\text{th}}$ and $j^{\text{th}}$ arguments;
\item $\mu$ is multiplication in $T_{\text{poly}}$, viewing the $f_i$ as cohomological degree 0 vector fields;
\item and $\pi$ is the projection $T_{\text{poly}} \twoheadrightarrow \Sym (V(-1))$. 
\end{itemize}
Clearly, $D_\Gamma$ is weight-homogeneous, so the Taylor components $U_n$ of the $L_\infty$ morphism are as well. Hence, we deduce

\begin{thm} \label{L infinity formality} Let $U$ be the Kontsevich $L_\infty$ formality morphism \cite{CF} restricted to $T_{\text{\emph{poly}}}^{\text{\emph{even}}} (V^\vee(1))(1)$. The composition
\[ \xymatrixcolsep{5pc}\xymatrix{
C^*_{\text{\emph{even}}} (\Sym (V(-1))(\text{\emph{wt}})(1) \ar[r]^-{(h^*)^{-1}} & C^*_{\text{\emph{even}}} (\Ext V)(1)\\
T_{\text{\emph{poly}}}^{\text{\emph{even}}} (V^\vee(1))(\text{\emph{wt}})(1) \ar[u]^-{U} & \ar[l]^-{\cong} \ar@{.>}[u] HH^*_{\text{\emph{even}}}(\Ext V)(1) } \]
is an $L_\infty$ quasi-isomorphism where $C^*(\Ext V)(1)$ and $HH^*(\Ext V)(1)$ are given their standard differential graded Lie structures. Therefore, when $\text{dim} \, V$ is even, $C^*(\Ext V)(1)$ is formal as an $L_\infty$ algebra.
\end{thm}

Formality for other types of algebra structures follows similarly, but there are some technical points in adapting results to the graded symmetric algebra. We give only a brief sketch. For any quasi-isomorphism of operads $G_\infty \to B_\infty$ \cite{Tamarkin}, the complexes $C^*_{\text{even}}(\Ext V)$ and $C^*_{\text{even}}(\Sym V(-1))(\text{wt})$ acquire isomorphic $G_\infty$ structures by Theorem \ref{B infinity}, reducing to the usual Gerstenhaber structures on cohomology. As for $L_\infty$ formality above, one needs to exhibit a $G_\infty$ quasi-isomorphism for $\Sym (V(-1))$ that is weight-homogeneous. For the ungraded symmetric algebra, such a morphism is demonstrated in \cite{CalaqueVDB}, which more generally proves affine equivariance. Alternatively, following \cite{Willwacher}, one could adapt the action of the colored operad $bigGra$ to the graded setting, imparting a $Br_\infty$ structure to $T_{\text{poly}}(V^\vee(1))$. This $Br_\infty$ structure and the action of Kontsevich graphs on the colored vector space $T_{\text{poly}}(V^\vee(1)) \oplus D_{\text{poly}}(V^\vee(1))$ are seen to be weight-homogeneous, so the arguments in \cite{Willwacher} imply $Br_\infty$ and $G_\infty$ formality for the exterior algebra on an even dimensional vector space.      

The odd dimensional case is where the exterior and symmetric algebras truly diverge. According to Theorem \ref{calculus}, the adjoint action of a volume form $\omega \in \Ext ^N V$ on Hochschild cohomology is trivial except on $HH^1(\Ext V)$. On the cochain level, however, any lift of $\omega$ pairs nontrivially with all graded components. This difference is already enough when $N =1$ to preclude $L_\infty$ formality. 

\begin{thm} \label{formality fails}
Let $\mathds{k}$ be any characteristic zero field, and suppose $|x| = 0$. There does not exist an $L_\infty$ quasi-isomorphism
\[ HH^*( \mathds{k}[x] / (x^2) )(1) \longrightarrow C^*( \mathds{k}[x]/(x^2) )(1) \]
where the left and right sides are given their standard differential graded Lie structures.
\end{thm}

\begin{proof}
Let $\mathfrak{g} = C^*(\mathds{k}[x] / (x^2)) (1)$, and suppose $f : H^*(\mathfrak{g}) \to \mathfrak{g}$ is an $L_\infty$ quasi-isomorphism with Taylor components $f_n: H^*(\mathfrak{g})^{\otimes n} \to \mathfrak{g}$. If $\partial$ is the unique derivation of $\mathds{k}[x] / (x^2)$ such that $\partial (x) = 1$, then by Theorem \ref{calculus}
\begin{equation*} 
H^{n-1}(\mathfrak{g}) \cong  
	\begin{cases}
		0 & \text{if} \; n < 0 \\
		\mathds{k}[x] / (x^2) & \text{if $n = 0$} \\
		\mathds{k} \cdot \partial^{n} & \text{if $n > 0$ is even} \\
		\mathds{k} \cdot x \partial^n & \text{if $n > 0$ is odd}.
	\end{cases} 
\end{equation*}
Since each component $H^{n-1}(\mathfrak{g})$ for $n > 0$ is one-dimensional, the cochain $\partial^{\otimes n}$ if $n$ is even or $x \partial^{\otimes n}$ if $n$ is odd, as specified in \eqref{cocycle evaluation}, is the unique cocycle in $\mathfrak{g}^{n-1}$ up to nonzero scaling and up to coboundary. 

Consider the elements $x \in H^{-1}(\mathfrak{g})$ and $x \partial^{2i-1} \in H^{2i-2}(\mathfrak{g})$ where $i \gg 1$. The first map $f_1$ is a quasi-isomorphism and thus takes the values, up to nonzero scaling, 
\[ f_1(x) = x + a, \; \; \; f_1 (x \partial^{2i-1}) = x \partial^{\otimes 2i-1} + d \gamma \]
for some $a \in \mathds{k}$ and $\gamma \in \mathfrak{g}^{2i-3}$. The induced map $\overline{f_1}: H^*(\mathfrak{g}) \to H^*(\mathfrak{g})$ is an isomorphism of Lie algebras, so
\[ x + a = \overline{f_1}(x) = \overline{f_1} ([x \partial, x]) = [\overline{f_1}(x \partial),  \overline{f_1}(x)] = [x \partial, x + a] = x, \]
implying $a = 0$.

 By definition of an $L_\infty$ morphism, $f$ satisfies the relation
\[ f_1([x, x \partial^{2i-1}]) = 2 [ f_1(x), f_1(x \partial^{2i-1}) ] + d f_2(x, x \partial^{2i-1}). \]
But $[x, x \partial^{2i-1}] = 0$ for $i > 1$ by Theorem \ref{calculus}, so 
\[ df_2 (x, x \partial^{2i-1}) = -2 [f_1(x), f_1(x \partial^{2i-1})] = -2 [x, x \partial^{\otimes 2i-1} + d \gamma] = 2( x \partial^{\otimes 2i-2} - [ x, d\gamma] ). \] 
By an identity of Gerstenhaber \cite{Gers}, the last term $[x, d \gamma]$ equals
\[ - d \gamma \circ x = - d(\gamma \circ x) + (-1)^{2i-2} \gamma \circ dx + \gamma \cup x - (-1)^{(2i-2) \cdot 0} x \cup \gamma = - d(\gamma \circ x), \]
the operation $\circ$ being the first brace operation (also known as the Gerstenhaber circle product). Moreover, it is straightforward to check that $ -(1/2) d (\partial^{\otimes 2i - 3}) = x \partial^{ \otimes 2i-2}$, so we deduce
\[ f_2 (x, x \partial^{2i-1}) =  - \partial^{\otimes 2i-3} + 2 \gamma \circ x + z \]
where $z \in \mathfrak{g}^{2i-4}$ is a cocycle. We may further specify $z$ up to scaling and up to coboundary, writing
\begin{equation} \label{f2}
f_2 (x, x \partial^{2i-1}) = - \partial^{\otimes 2i-3} + 2 \gamma \circ x +  b x \partial^{\otimes 2i-3} + d\rho 
\end{equation}
where $b \in \mathds{k}$ and $\rho \in \mathfrak{g}^{2i-5}$.

At the next level, $f$ satisfies the relation
\begin{align*}
f_2 ([x,x], x \partial^{2i-1}) - 2f_2([x, x\partial^{2i-1}], x) & \\
=  2 [ f_1(x), f_2(x, x \partial^{2i-1}) ] & - [ f_1(x \partial^{2i-1}), f_2(x, x) ]  + d f_3(x, x, x \partial^{2i-1}).  
\end{align*} 
But $[x, x] = [x, x \partial^{2i-1}] = 0$, and for degree reasons $f_2(x, x) = 0$. Thus, this equation reduces to
\[ [x, f_2(x, x\partial^{2i-1}) ] = \frac{1}{2} d f_3 (x, x, x \partial^{2i-1}), \]
expressing the requirement that the Massey triple product $\langle x, x, x \partial^{2i-1} \rangle$ is zero in $H^*(\mathfrak{g})$. For a contradiction, however, we show that $[x, f_2(x, x\partial^{2i-1}) ]$ cannot be a coboundary. Using the expression \eqref{f2} above, observe
\begin{eqnarray*} 
[ x, f_2(x, x \partial^{2i-1}) ] & = & [x, - \partial^{\otimes 2i-3} + 2  \gamma \circ x + b x \partial^{\otimes 2i-3} + d\rho ] \\
& = & \partial^{ \otimes 2i-4} - ( \gamma \circ x ) \circ x - b x \partial^{\otimes 2i-4} - d \rho \circ x  \\
& = & \partial^{ \otimes 2i-4} - ( \gamma \circ x ) \circ x + \frac{1}{2} b d( \partial^{\otimes 2i-5}) - d (\rho \circ x).
\end{eqnarray*}
So it suffices to show $ \partial^{\otimes 2i-4} - (\gamma \circ x) \circ x$ is not a coboundary. Evaluating at $x^{\otimes 2i-4}$, we compute
\[ \partial^{\otimes 2i-4}(x^{\otimes 2i-4}) - (\gamma \circ x) \circ x \, (x^{\otimes 2i-4}) = -1 - 0. \]
On the other hand, for any $\mu \in \mathfrak{g}^{2i-6}$, $d\mu (x^{\otimes 2i-4})$ is a multiple of $x$; consequently, $d \mu $ cannot equal $ \partial^{\otimes 2i-4} - (\gamma \circ x) \circ x$. 
\end{proof}

We conjecture that the statement and proof generalize to any odd dimension.

To examine implications for deformation theory, recall that a Maurer-Cartan element of a differential graded Lie algebra $\mathfrak{g}$ is an element $\gamma \in \mathfrak{g}^1$ satisfying
\[ d \gamma + \frac{1}{2} [\gamma, \gamma] = 0. \]
We write $MC(\mathfrak{g})$ for the set of gauge equivalence classes of Maurer-Cartan elements (see, for example, \cite{Kont}). 

Let $R = \mathbb{R} \oplus R_+$ be a complete, augmented, and commutative $\mathbb{R}$-algebra, e.g., $R = \mathds{\mathbb{R}}[[\hbar]]$. For an associative algebra $A$ in cohomological degree 0, formal deformations of $A$ over $R$ are the same as Maurer-Cartan elements of the differential graded Lie algebra $C^*(A)(1) \hat{\otimes} R_+$, the symbol $\hat{\otimes}$ denoting the completed tensor product. A standard fact is that an $L_\infty$ quasi-isomorphism between two differential graded Lie algebras $\mathfrak{g}$ and $\mathfrak{h}$ induces a bijection $MC(\mathfrak{g} \hat{\otimes} R_+) \leftrightarrow MC(\mathfrak{h} \hat{\otimes} R_+)$. So for the exterior algebra, the $L_\infty$ quasi-isomorphism of Theorem \ref{L infinity formality} induces a bijection between gauge equivalence classes of formal deformations and
\[ MC( HH_{\text{even}}^*(\Ext V)(1) \hat{\otimes} R_+) = MC(HH^*(\Ext V)(1) \hat{\otimes} R_+ ). \]
But by the isomorphism in Theorem \ref{calculus}, this set is in bijective correspondence with
\[ MC (T_{\text{poly}}^{\text{even}} (V^\vee (1))(\text{wt})(1) \hat{\otimes} R_+). \]

Let us examine this set more closely. Up to gauge equivalence, an element is of the form $\sum_{i = 1}^\infty \gamma_i r_i$ where, given $n > 0$, $r_i$ is an element of $R_+^n$ for sufficiently large $i$ and $\gamma_i \in T_{\text{poly}}^{\text{even}}(V^\vee(1))(\text{wt})^2$ is a sum of even-weight bivectors, i.e., 
\[\xi_{j_1} \dots \xi_{j_t}  \frac{\partial}{\partial \xi_{i_1}} \frac{\partial}{\partial \xi_{i_2}}, \; \; t \equiv 0 \; \text{mod} \; 2  \]
in the notation of Remark \ref{Hochschild koszul duality}. The Maurer-Cartan condition is that the self-bracket under the $R_+$-linearly extended Schouten-Nijenhuis bracket is zero,
\[ [\sum_{i = 1}^\infty \gamma_i r_i, \sum_{i = 1}^\infty \gamma_i r_i] = \sum_{i,j = 1}^\infty [\gamma_i, \gamma_j]_{\text{SN}} r_i r_j = 0. \]
We say such elements are formal Poisson since a bivector $\gamma \in T_{\text{poly}}(V)$ satisfying $[\gamma, \gamma]_{\text{SN}} = 0$ defines a Poisson bracket on $\Sym (V^\vee)$. Through the isomorphism \eqref{koszul interchange}, the $\gamma_i$ correspond to polyvectors in $T_{\text{poly}}^{\text{even}}(V)(\text{wt})$ with quadratic coefficients, 
\[ \xi_{j_1} \dots \xi_{j_t}  \frac{\partial}{\partial \xi_{i_1}} \frac{\partial}{\partial \xi_{i_2}}  \longleftrightarrow x_{i_1}^\vee x_{i_2}^\vee \frac{\partial}{\partial x_{j_1}^\vee}  \dots\frac{\partial}{\partial x_{j_t}^\vee}, \]
and the resulting formal polyvector in $T_{\text{poly}}^{\text{even}}(V)(\text{wt})(1) \hat{\otimes} R_+$ still has self-bracket equal to 0. 

\begin{cor} \label{deformation theory}
For $V$ of any dimension, formal deformations of $\Ext V$ over $R$ are in one-to-one correspondence with formal Poisson bivectors in
\[ T_\text{\emph{poly}}^\text{\emph{even}} (V^\vee(1))(\text{\emph{wt}})(1) \hat{\otimes} R_+, \]
up to gauge equivalence. Under the Koszul duality between $\Sym (V^\vee)$ and $\Sym(V(-1))$, these correspond to formal Poisson polyvectors with quadratic coefficients in
\[   T_\text{\emph{poly}}^\text{\emph{even}} (V)(\text{\emph{wt}})(1) \hat{\otimes} R_+. \]
\end{cor}

Therefore, just as for the ordinary symmetric algebra, deformation theory of the exterior algebra can be studied in terms of formal Poisson structures.

\appendix

\section{Hochschild (co)homology of a graded algebra}

The literature contains different sign conventions for Hochschild (co)homology of a graded algebra. The subject can be approached on the one hand from the homological algebra of ordinary graded objects (e.g., \cite{kassel}) or, on the other, from the theory of differential graded objects (e.g., \cite{Abb}). The following establishes a relationship between the two from a general point of view, paying careful attention to sign rules. Comparing graded modules to differential graded modules and comparing their derived functors have been done before (see e.g. \cite{beck}, \cite{BMR}, \cite{PolVan}), but we could not find a discussion of totalization of Hochschild (co)homology for non-perfect bimodules.

Let $B$ be a $\mathbb{Z}$-graded associative $\mathds{k}$-algebra and $B\mbox{-}gr$ the abelian category whose objects are $\mathbb{Z}$-graded left $B$-modules and morphisms are degree preserving $B$-linear maps. Then consider the category of chain complexes $Ch(B\mbox{-}gr)$. An object in this category is a bicomplex $(M^{*,*}, d)$ where the first index is cohomological degree and the second is internal degree; the horizontal differential has degree $(+1, 0)$ and the vertical differential is trivial. The corresponding category for right $B$-modules, $Ch(gr\mbox{-}B)$, is defined in the same way.

Let $B^{dg}$ be the algebra $B$ thought of as a differential graded algebra. Let $B^{dg}\mbox{-}mod$ be the category whose objects are differential graded left $B^{dg}$-modules and morphisms are degree preserving $B^{dg}$-linear maps. In particular, an object is a chain complex $(M^*, d)$ satisfying $d(bm) = (-1)^{|b|} bd(m)$ for all $b \in B^{dg}$. The corresponding category for right $B^{dg}$-modules, $mod\mbox{-}B^{dg}$, is defined in the same way except the differential satisfies $d(mb) = d(m) b$ for all $b \in B^{dg}$. 

There are two functors by which to obtain an object in $B^{dg}\mbox{-}mod$ from an object $M \in Ch(B \mbox{-} gr)$. First, there is direct sum totalization, 
\[ \Tplus (M)^i : = \bigoplus_{j \in \mathbb{Z}} \Sigma^{j} (M^{j, i-j}) \]
where $\Sigma$ denotes suspension of the internal degree. The $B^{dg}$-module structure is defined by
\[ b \, s^{j} m = (-1)^{j |b|} s^{j} bm \]
and the differential by 
\[ d(s^{j} m) : = s^{j + 1} d(m) \]
so as to satisfy the needed identity $d(b \, s^{j}m) = (-1)^{|b|} b d(s^{j}m)$ \cite{beck}. Alternatively, there is direct product totalization, 
\[ \Tprod (M)^i = \prod_{j \in \mathbb{Z}} \Sigma^{j} (M^{j, i-j}), \]
with $B^{dg}$-module structure and differential defined as above in each coordinate. The two totalization functors are defined the same way for $Ch(gr\mbox{-}B)$ except there is no sign in the right action of $B^{dg}$.

Let $Ch_{\leq 0}(B\mbox{-}gr)$ be the full subcategory of $Ch(B\mbox{-}gr)$ consisting of chain complexes concentrated in non-positive cohomological degree. In the standard projective model structure on $Ch_{\leq 0}(B\mbox{-}gr)$ \cite{hovey}, cofibrant objects are complexes of projective graded $B$-modules. If $P$ is such an object, then $\Tplus(P)$ is a retract of a semifree $B^{dg}$-module, which is cofibrant in the standard projective model structure on $B^{dg}\mbox{-}mod$ \cite{BMR}. This leads to the following comparison:

\begin{prop} \label{equiv}
Let $M \in B\mbox{-}gr$ and $N \in gr\mbox{-}B$. In $\mathds{k}^{dg}\mbox{-}mod$, there is a quasi-isomorphism
\[ \Tplus (N \otimes^{L}_B M) \longrightarrow \Tplus (N) \otimes^{L}_{B^{dg}} \Tplus(M). \]
\end{prop}

\begin{proof}
Let $P^{*,*} \in Ch_{\leq 0}(B\mbox{-}gr)$ be a projective resolution of $M$. Define a map 
\[ \varphi: \Tplus (N \otimes_B P) \to \Tplus(N) \otimes_{B^{dg}} \Tplus(P) \]
that on homogeneous elements of $\Sigma^j N \otimes_B P^{j, *}$ evaluates as
\[ \varphi: s^{j}(n \otimes_B p) \mapsto (-1)^{j |n|} n \otimes_{B^{dg}} s^{j} p. \]
Observe that
\[ \varphi (s^j (n \otimes_B bp)) = (-1)^{j |n|} n \otimes_{B^{dg}} s^j (bp) = (-1)^{j (|n| + |b|)} nb \otimes_{B^{dg}} s^j p = \varphi (s^j(nb \otimes_B p)), \]
so $\varphi$ is well-defined. It is clearly a levelwise isomorphism, so we have only to show that it is compatible with the differentials:
\begin{align*} 
\varphi \circ d(s^{j} (n \otimes_B p)) & =  \varphi (s^{j + 1}(n \otimes_B dp)) = (-1)^{(j+1)|n|} n \otimes_{B^{dg}} s^{j+1} dp, \\
d \circ \varphi (s^{j}(n \otimes_B p)) & = d( (-1)^{j |n|} n \otimes_{B^{dg}} s^{j} p ) = (-1)^{j|n| + |n|} n \otimes_{B^{dg}} s^{j+1} dp. \qedhere
\end{align*}
\end{proof} 

There is an analogous statement for derived $Hom$. For any $M, N \in B\mbox{-}gr$, let $Hom_B(M,N)$ be the graded vector space spanned by homoegeneous $B$-linear maps
\[ Hom_B(M,N) : = \bigoplus_{n \in \mathbb{Z}} Hom_{B\mbox{-}gr} (M, N[n]). \]
Similarly, for any $M, N \in B^{dg}\mbox{-}mod$, let $Hom_{B^{dg}}(M,N)$ be the differential graded vector space spanned by homogeneous $B^{dg}$-linear maps
\[ Hom_{B^{dg}}(M, N) : = \bigoplus_{n \in \mathbb{Z}} Hom_{B^{dg}\mbox{-}mod} (M, N(n)), \; \; d f = d_N f - (-1)^{|f|} f d_M. \]

\begin{prop} \label{coh equiv}
Let $M, N \in B\mbox{-}gr$. In $\mathds{k}^{dg}\mbox{-}mod$, there is a quasi-isomorphism
\[ \Tprod \, RHom_B (M, N) \longrightarrow R Hom_{B^{dg}} (\Tplus(M), \Tplus (N) ). \]
\end{prop}

\begin{proof} 
Let $P^{*,*} \in Ch_{\leq 0}(B\mbox{-}gr)$ be a projective resolution of $M$. Define a map
\[ \psi: \Tprod \Big( \bigoplus_{j \geq 0} Hom_B (P^{\mbox{-}j, *}, N)(-j) \Big) \to Hom_{B^{dg}} ( \Tplus(P), \Tplus(N)) \]
that on a homogeneous element of degree $i$ evaluates as
\[ \psi: \prod_{j \geq 0} s^j f_j \mapsto \Big( \tilde{f} : s^{-k} p \mapsto (-1)^{k (i - k)} f_k(p) \Big). \]
Observe
\begin{align*}
\tilde{f}( b s^{-k} p ) = (-1)^{-k |b|} \tilde{f}(s^{-k} bp) & = (-1)^{-k |b| + k(i - k)} f_k(bp) \\
& = (-1)^{-k |b| + k(i - k) + (i-k)|b|} b f_k(p) = (-1)^{i |b|} b \tilde{f} (s^{-k} p ), 
\end{align*}
so $\tilde{f}$ is indeed $B^{dg}$-linear. The map $\psi$ is clearly a levelwise isomorphism, so we have only to show that it is compatible with differentials: 
\begin{align*} \psi \circ d (\prod s^j f_j)(s^{-k} p) & = \psi (\prod s^{j+1} d f_j) (s^{-k} p) \\
& = (-1)^{k(i + 1 - k)} df_{k-1}(p) \\
& = (-1)^{k(i + 1 -k) + k} f_{k-1}( dp ), \\ 
d \circ \psi (\prod s^j f_j)(s^{-k} p) & = d \tilde{f} (s^{-k }p) \\
& = (-1)^{i + 1} \tilde{f}(d( s^{-k} p)) \\
& = (-1)^{i + 1} \tilde{f} (s^{-k +1} dp) \\
& = (-1)^{i + 1 + (k-1)(i - k + 1)} f_{k-1} (dp). \qedhere
 \end{align*}
 \end{proof} 

Letting $B = A^e$, $M = N = A$, and $P = \text{Bar}(A)$ in the above propositions, we have isomorphisms
\begin{eqnarray*} 
\varphi: \Tplus (C_{*,*} (A)) & \longrightarrow & A^{dg} \otimes_{(A^e)^{dg}} \Tplus(\text{Bar}(A)) \\
\psi: \Tprod (\widetilde{C}^{*,*}(A)) & \longrightarrow & Hom_{(A^e)^{dg}} (\Tplus(\text{Bar}(A)), A^{dg}).
\end{eqnarray*}
Furthermore, there is an isomorphism between the totalized bigraded bar resolution and the differential graded bar resolution
\begin{align*}
& q: \Tplus (\text{Bar}(A)) \to \text{Bar}^{dg}(A), \\
& q( s^n a_0[a_1 | \dots | a_n]a_{n+1}) = (-1)^{n |a_0| + (n-1)|a_1| + \dots + |a_{n-1}| } a_0[sa_1| \dots |sa_n] a_{n+1}. &
\end{align*}
Let $Q: Hom(\Tplus(\text{Bar}(A)), A^{dg}) \to Hom(\text{Bar}^{dg}(A), A^{dg})$ be the precomposition $Q(f) = f \circ q^{-1}$. Then we have the desired isomorphisms
\begin{eqnarray}
F &: = & (id_{A^{dg}} \otimes_{(A^e)^{dg}} q) \circ \varphi: \Tplus (C_{*,*}(A)) \longrightarrow C_*(A) \label{iso 1} \\
G &: =& Q \circ \psi: \Tprod (\widetilde{C}^{*,*}(A))  \longrightarrow \widetilde{C}^*(A). \label{iso 2} \nonumber
\end{eqnarray}
 
However, the computations in the body of the paper pertain to $A$ of the form \eqref{form} and to the subcomplexes of Hochschild cochains spanned by weight-homogeneous elements, which we labeled by $C^{*,*}(A)$ and $C^*(A)$. In the above formalism, graded $B$-modules and $B^{dg}$-modules are replaced by bigraded $B$-modules and graded $B^{dg}$-modules; the functors $\otimes_B$, $\otimes_{B^{dg}}$, $Hom_B$, and $Hom_{B^{dg}}$ are bigraded. Hochschild homology is automatically graded by weight. If in $A$ the internal degree of the generating vector space $V$ is 0, then the functors $\Tplus$ and $\Tprod$ are the identity functors, and the map $G$ restricts to the identity of $C^*(A)$. If the internal degree of the generators is a nonzero integer, then $\widetilde{C}^{*,*}(A) = C^{*,*}(A)$, and $G$ identifies the direct product totalization with $\widetilde{C}^*(A)$, which is the completion of $C^*(A)$ with respect to the weight grading. The subspace mapping isomorphically to $C^*(A)$ under $G$ is precisely the direct sum totalization,
\begin{equation}
G|_{\Tplus} : \Tplus (C^{*,*}(A)) \longrightarrow C^*(A) \label{iso 3}. 
\end{equation}
It is easy to check that the functor $\Tplus: Ch(\mathds{k}\mbox{-}gr) \to \mathds{k}^{dg}\mbox{-}mod$  exchanges shifts in internal degree with homological shifts, including shearing,
\[ \Tplus( M [\text{wt}] ) = \Tplus(M)(\text{wt}). \]

\bibliographystyle{plain}

\bibliography{Hochschildcohomologyofexterioralgebra}

\end{document}